\documentclass[11pt]{amsart}

\title[The Goeritz groups of $(1,1)$-decompositions]
{The Goeritz groups of $(1,1)$-decompositions}

\author{Yuya Koda}
\address{
Department of Mathematics, Hiyoshi Campus, Keio University, 4-1-1, Hiyoshi, Kohoku, Yokohama, 223-8521, Japan~ \slash ~ 
International Institute for Sustainability with Knotted Chiral Meta Matter (WPI-SKCM$^2$), Hiroshima University, 1-3-1 Kagamiyama, Higashi-Hiroshima, 739-8526, Japan}
\email{koda@keio.jp}

\author{Yuki Tanaka}
\address{
Department of Mathematics, Hiroshima University, 1-3-1 Kagamiyama, Higashi-Hiroshima, 739-8526, Japan}
\email{tanakayuki3.15@icloud.com}

\thanks{
Y. K. is supported by JSPS KAKENHI Grant Numbers JP20K03588, 
JP21H00978 and JP23H05437. 
}

\usepackage{amsthm}
\usepackage{mathrsfs}
\usepackage{latexsym}

\usepackage[dvipdfmx]{graphicx}
\usepackage[dvips]{psfrag}
\usepackage[dvips]{color}
\usepackage{xypic}
\usepackage[abs]{overpic}
\usepackage{caption}
\usepackage[all]{xy}
\usepackage[normalem]{ulem}
\usepackage{tikz}

\theoremstyle{plain}
\newtheorem*{theorem*}{Theorem}
\newtheorem*{lemma*} {Lemma}
\newtheorem*{corollary*} {Corollary}
\newtheorem*{proposition*}{Proposition}
\newtheorem*{conjecture*}{Conjecture}
\newtheorem{theorem}{Theorem}[section]
\newtheorem{lemma}[theorem]{Lemma}
\newtheorem{corollary}[theorem]{Corollary}
\newtheorem{proposition}[theorem]{Proposition}

\theoremstyle{remark}

\newtheorem*{definition}{Definition}
\newtheorem*{claim*}{Claim}

\theoremstyle{definition}

\newtheoremstyle{citing}
  {}
  {}
  {\itshape}
  {}
  {\bfseries}
  {.}
  {.5em}
  {\thmnote{#3}}

\theoremstyle{citing}

\newcommand{\NN}{\mathbb{N}}
\newcommand{\ZZ}{\mathbb{Z}}
\newcommand{\RR}{\mathbb{R}}

\newcommand{\MCG}{\mathrm{MCG}_+}

\newcommand{\Nbd}{\operatorname{Nbd}}
\newcommand{\Cl}{\operatorname{Cl}}

\def\vtil{\widetilde{V}_1}



\begin{document}

\maketitle

\begin{abstract}
A $(g, n)$-decomposition of a link $L$ in a closed orientable $3$-manifold $M$ is a decomposition of $M$ by a closed orientable surface of genus $g$ 
into two handebodies each intersecting the link $L$ in $n$ trivial arcs. 
The Goeritz group of that decomposition is then defined to be the group of isotopy classes of 
orientation-preserving homeomorphisms of the pair $(M, L)$ that preserve the decomposition. 
We compute the Goeritz groups of all $(1,1)$-decompositions.
\end{abstract}

\vspace{1em}

\begin{small}
\hspace{2em}  \textbf{2020 Mathematics Subject Classification}: 57K20, 57K10


\hspace{2em} 
\end{small}

\section*{Introduction}

Let $L$ be a link in a closed orientable $3$-manifold $M$. 
We say that $(M, L; \Sigma)$ is a \emph{$(g, n)$-decomposition} of $(M, L)$ for some non-negative integers $g$ and $n$ 
if $(M; \Sigma)$ is a Heegaard splitting of genus $g$ of $M$, that is, $\Sigma$ decomposes $M$ into two 
handlebodies $V_1$ and $V_2$ of genus $g$, and the intersection of $L$ and $V_i$ consists of $n$ trivial arcs in $V_i$ for 
each $i=1,2$. 
The mapping class group of this decomposition, defined as the group of isotopy classes of 
orientation-preserving homeomorphisms of $M$ preserving the triple $(M, V_1, L)$ setwise, 
is called the \textit{Goeritz group} of $(M, L; \Sigma)$,  and it is denoted by $\mathcal{G}(M, L; \Sigma)$.  
The properties of this group are often effectively described by using the \emph{Hempel distance}, 
a non-negative integer valued complexity of $(g, n)$-decompositions. 
In fact, the following is established by Iguchi and the first-named author.

\begin{theorem}[Iguchi--Koda~\cite{Iguchi-Koda}]
\label{thm: Iguchi-Koda}
Let $(M, L; \Sigma)$ be a $(g, n)$-decomposition of a link $L$ in a closed orientable $3$-manifold $M$, where 
$(g, n) \neq (0, 1), (0, 2), (1, 1)$. 
If $d(M, L; \Sigma) \geq 6$, then the Goeritz group $\mathcal{G}(M, L; \Sigma)$ is a finite group. 
\end{theorem}

Among the cases of $(g, n)$ excluded from the above theorem, the Goeritz groups for the cases $(g, n) = (0, 1), (0, 2)$ can be computed easily, as 
explained in Hirose--Iguchi--Kin--Koda~\cite{Hirose-Iguchi-Kin-Koda}. 
In this paper, we focus on the remaining case, namely, the case $(g, n) = (1,1)$. 
The following is the main theorem, where the definitions of the specific elements of the Goeritz groups such as 
$\alpha, \beta$, etc., as well as a specific knot $K = K_{p/q}$ and surface $\Sigma = \Sigma_{p/q}$, 
will be given in Sections~\ref{sec: Mapping class groups of a trivial 1-tangle in a solid torus} and \ref{sec: Proof of Main Theorem}.

\begin{theorem}
\label{shukekka}
Let $(M, K; \Sigma)$ be a $(1,1)$-decomposition of a knot $K$ in 
a closed orientable $3$-manifold $M$. 
Then we have the following. 
\begin{enumerate}
  \item 
  If $M \neq S^2 \times S^1$ and $K$ is the trivial knot, then 
  we have $\mathcal{G}(M, K; \Sigma) = \ZZ / 2 \ZZ \langle \alpha \rangle \times \ZZ \langle \beta \rangle$. 
  \item 
  If $M = S^2 \times S^1$ and $K$ is the trivial knot, then we have 
  $\mathcal{G}(M, K; \Sigma) = \ZZ / 2 \ZZ \langle \alpha \rangle \times \ZZ \langle \tau \rangle \times \langle \beta, \gamma \mid \gamma^2 = 1 \rangle$.
  \item 
  If $M = S^2 \times S^1$ and $K$ is the core knot, 
  then we have 
  $\mathcal{G}(M, K; \Sigma) = \ZZ / 2 \ZZ \langle \alpha \rangle \times \ZZ \langle \tau \rangle \times \ZZ \langle \tau' \rangle$.
  \item 
  Let $p$ and  $q$ be coprime inters, where $p$ is a non-zero even number. 
For  $M = L(p, q)$, $K = K_{p/q}$, $\Sigma = \Sigma_{p/q}$, 
we have $\mathcal{G}(M, K; \Sigma) = \ZZ / 2 \ZZ \langle \alpha \rangle \times \ZZ / 2 \ZZ \langle \gamma \rangle$. 
  \item 
  Otherwise, we have $\mathcal{G}(M, K; \Sigma) = \ZZ / 2 \ZZ \langle \alpha \rangle$. 
\end{enumerate}
\end{theorem}

In general, even if we fix $(M, L)$, $g \in \NN \cup \{0\}$ and $n \in \NN$, 
the Goeritz group of a $(g, n)$-decomposition of  $L \subset M$ may depend on the specific choice of the $(g , n)$-decomposition. 
However, the above theorem implies that for  $(1, 1)$-decompositions, the Goeritz groups depend only on 
the ambient 3-manifold $M$ and the knot $K$. 

Combining this theorem with the result of Saito~\cite{Saito} on knots with $(1, 1)$-decompositions with Hempel distance at most $1$ (see Theorem \ref{kyoridouti} below),  we see that the assertion of Theorem \ref{thm: Iguchi-Koda} remains valid for $(g, n) = (1, 1)$.
In fact, the following stronger result follows.

\begin{corollary}
Let $(M, K; \Sigma)$ be a $(1, 1)$-decomposition of a knot $K$ in a closed orientable $3$-manifold $M$. 
Then the Goeritz group $\mathcal{G}(M, K; \Sigma)$ is a finite group if and only if  
$d(M, K; \Sigma) \geq 2$. 
\end{corollary}

\section{Preliminaries}
\label{sec: Preliminaries}

Throughout the paper, we will work in the piecewise linear category. 
Any curves (resp. surfaces) in a surface (resp. a $3$-manifold) are always assumed to be properly embedded, and their intersection is transverse and minimal up to isotopy, unless otherwise mentioned. 
For convenience, we will often not distinguish 
curves, surfaces, homeomorphisms, e.t.c.  
from their isotopy classes in their notation. 
For a subspace $Y$ of a space $X$,  
$\Nbd(Y; X)$, or simply $\Nbd(Y)$, will denote a regular neighborhood of $Y$, and $\Cl(Y)$ the closure of $Y$. 
The number of components of $X$ is denoted by $|X|$. 
For a $1$-submanifold $L$ of a 3-manifold $M$, we put $E(L) = \Cl ( M \setminus \Nbd (L))$ and call it the \textit{exterior} of $L \subset M$. 
We apply maps or mapping classes from right to left, i.e., the
product fg means that g is applied first.

Let $M$ be a closed orientable $3$-manifold, and let $\Sigma$ be a closed orientable surface embedded in $M$. 
Then, the pair $(M; \Sigma)$ is called a \emph{Heegaard splitting} of $M$ if 
there exist handlebodies $V_1$ and $V_2$ in $M$ such that 
$M = V_1 \cup V_2$ and $V_1 \cap V_2 = \partial V_1 = \partial V_2 = \Sigma$. 
The surface $\Sigma$ here is called a \emph{Heegaard surface} of $M$, and 
the \emph{genus} of the splitting is defined to be the genus of $\Sigma$. 
This Heegaard splitting is sometimes denoted by $(V_1, V_2; \Sigma)$ as well.

Let $V$ be a handlebody of genus $g$. 
The union $T_1 \cup T_2 \cup \cdots \cup T_n$ of $n$ arcs properly embedded in $V$ is called an \emph{$n$-tangle}.
An $n$-tangle $T_1 \cup T_2 \cup \cdots \cup T_n$ is said to be \emph{trivial} if for each $i$, 
there exists a disk $D_i \subset V$ such that $T_i \subset \partial D_i$, $\partial D_i - T_i \subset \partial V$ and $D_i \cap D_j = \emptyset$ ($i \neq j$). 
Each disk $D_i$ here is called a \emph{canceling disk} for $T_i$.

Let $(V_1, V_2; \Sigma)$ be a genus-$g$ Heegaard splitting of a closed orientable  $3$-manifold $M$, and 
let $L$ be a link in $M$. 
We call $(M, L; \Sigma)$ a \emph{$(g,n)$-decomposition}, or simply  a \emph{bridge decomposition}, of $(M, L)$ if 
$V_i \cap L$ is the trivial $n$-tangle in $V_i$ for $i=1,2$. 
This $(g,n)$-decomposition is sometimes denoted by $(V_1, V_2, L; \Sigma)$ as well. 
It is easy to see that every link admitting a $(g,n)$-decomposition admits a $(g+1, n-1)$ decomposition. 
Figure~\ref{figure:figure-eight} shows an example of that in the case of $(g,n) = (0,2)$.

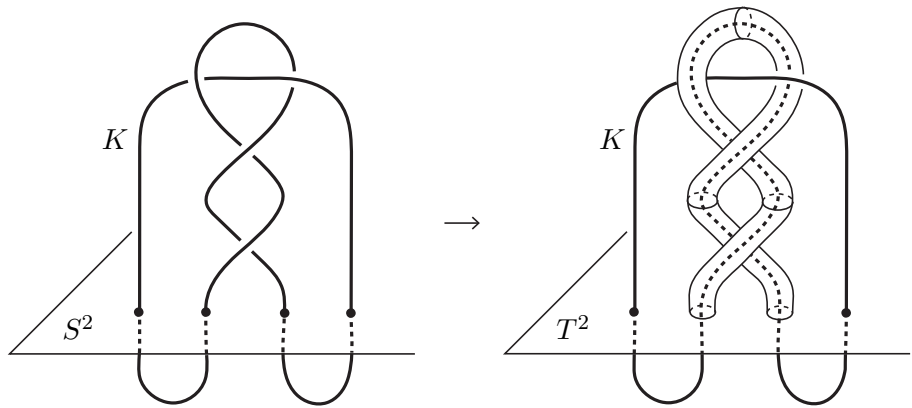
\begin{figure}[htbp]
\centering\includegraphics[width=12cm]{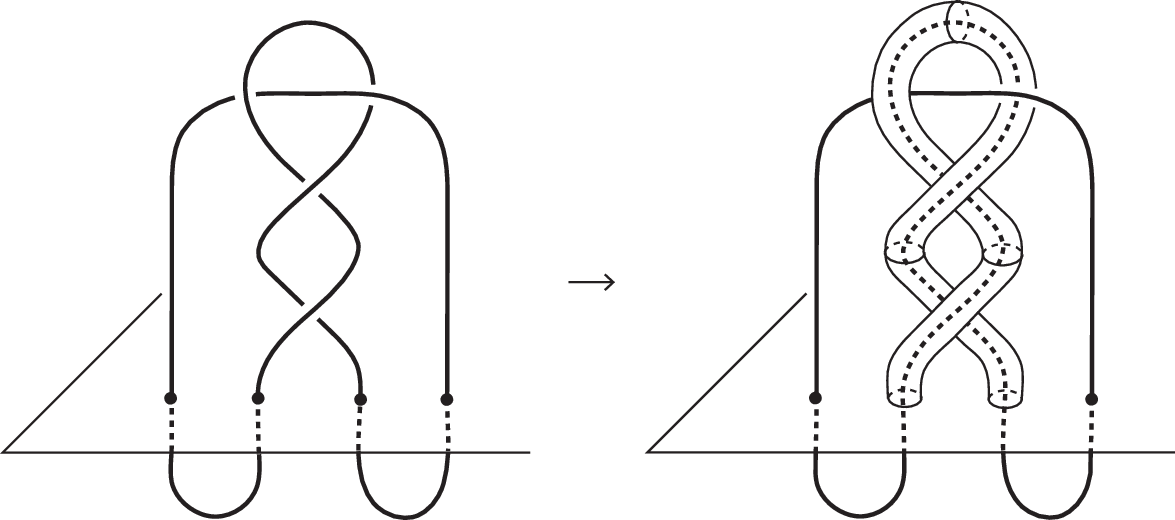}
\begin{picture}(400,0)(0,0)
\put(50,38){$S^2$}
\put(65,110){$K$}

\put(253,110){$K$}
\put(237,38){$T^2$}
\end{picture}
\caption{From a $(0,2)$-decomposition of the figure-eight knot to 
its $(1, 1)$-decomposition.}
\label{figure:figure-eight}
\end{figure}

Let $M$ be a closed orientable $3$-manifold. 
A knot $K$ in $M$ is called the \emph{trivial knot} if there exists a disk $D$ in $M$ 
such that $\partial D = K$. 
A knot $K$ is called a \emph{torus knot} if there exists a Heegaard surface $\Sigma$ of genus $1$ such that  $K \subset \Sigma$.
A knot $K$ is called a \emph{core knot} if the exterior $E(K)$ is a solid torus. 
Note that core knots are also torus knots. 
When $M=S^3$, the trivial knot is the unique core knot. 
Every torus knot admits a $(1, 1)$-decomposition as shown in Figure~\ref{figure:trefoil}. 

\begin{figure}[htbp]
\centering\includegraphics[width=10cm]{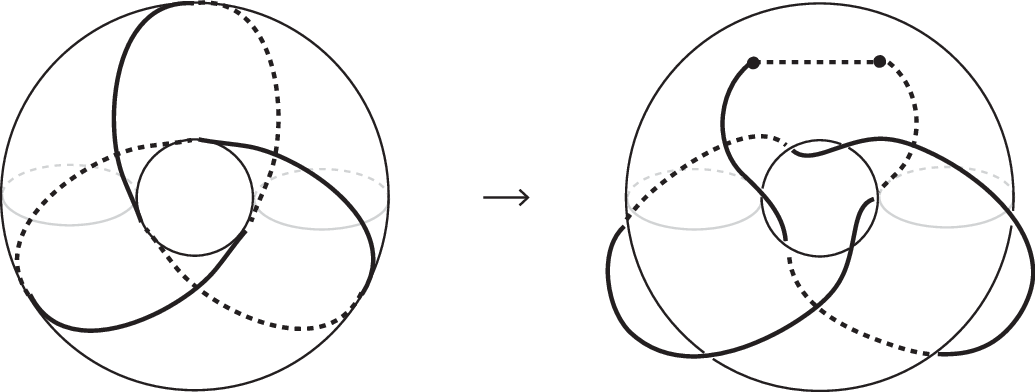}
\caption{Every torus knot admits a $(1, 1)$-decomposition.}
\label{figure:trefoil}
\end{figure}

Let $L$ and $L'$ be links in a closed orientable $3$-manifold $M$. 
Let $(M, L; \Sigma)$ and $ (M, L' ; \Sigma')$ be their bridge decompositions. 
We say that they are \emph{equivalent} if there exists an orientation-preserving homeomorphism $f$ of 
$M$ such that $f(\Sigma) = \Sigma'$ and $f(L) = L'$. 
The following results were shown by 
Morimoto~\cite{Mor89} and Kobayashi--Saeki~\cite{Kobayashi-Saeki}. 
See also Cho--Koda~\cite{Cho-Koda}.

\begin{theorem}
\begin{enumerate}
\item The $(1, 1)$-decomposition of a torus knot in a closed orientable $3$-manifold is unique up to equivalence. 
\item The $(1, 1)$-decomposition of a $2$-bridge $($in $S^3$$)$ knot is unique up to equivalence.
\end{enumerate}
\end{theorem}

Let $\Sigma_{g, k}$ be a closed orientable surface of genus $g$ with $k$ marked points.
A simple closed curve $c$ in $\Sigma_{g, k}$ is said to be \emph{essential} if $c$ 
does not bound a disk with at most one marked point in $\Sigma_{g, k}$.
The \emph{curve graph} $\mathcal{C}(\Sigma_{g, k})$ of $\Sigma_{g, k}$ is defined to be a graph whose vertex set is the set of isotopy classes of essential simple closed curves in $\Sigma_{g, k}$, where two vertices $c_1$ and $c_2$ span an edge if they are distinct and they have disjoint representatives. 
A necessary and sufficient condition for $\mathcal{C}(\Sigma_{g, k})$ to be non-empty and connected is $3g-4+k > 0$. 
In this case, the metric $d_{\mathcal{C}(\Sigma_{g, k})}$ on $\mathcal{C}(\Sigma_{g, k})$ is given by the path length, where each edge has length one. 

Let $(g, n) \neq (0, 1), (0, 2)$, and let $(M, L; \Sigma)$ be a $(g, n)$-decomposition of a link $L$  in a closed orientable $3$-manifold $M$. 
We identify the pair $(\Sigma, \Sigma \cap L)$ with a closed orientable surface $\Sigma_{g, 2n}$ with $2n$ marked points. 
Let $\mathcal{D}(V_i - L)$ ($i=1,2$) denote the subset of $\mathcal{C}(\Sigma_{g, 2n})$ consisting of 
the vertices whose representatives bound disks in $V_i - L$. 
Then, the \emph{distance} $d(M, L; \Sigma)$ of the decomposition $(M, L; \Sigma)$ is defined by 
\[ d(M, L; \Sigma) := d_{\mathcal{C}(\Sigma_{g, 2n})}(\mathcal{D}(V_1 - L), \mathcal{D}(V_2 - L)) .\]

The following theorem by Saito~\cite{Saito} classifies all knots that admit $(1,1)$-decompositions of distance 
at most one. 

\begin{theorem}[Saito~\cite{Saito}]\label{kyoridouti}
Let $(M, K; \Sigma)$ be a $(1, 1)$-decomposition of a knot $K$ in a closed orientable $3$-manifold $M$. 
\begin{enumerate}
\item We have $d(M, K; \Sigma) = 0$ if and only if $K$ is the trivial knot. 
\item We have $d(M, K; \Sigma) = 1$ if and only if $M = S^2 \times S^1$ and $K$ is the core knot. 
\end{enumerate}
\end{theorem}

Finally, we review the definition of the Goeritz group of bridge decompositions. 
Let $M$ be a closed orientable $3$-manifold, and let $X_i$ ($i = 1, \ldots , n$) be a subspace of $M$.
The \emph{mapping class group} of $(M, X_1, \ldots , X_n)$, denoted by $\MCG(M, X_1,\ldots , X_n)$, 
is defined as the group of  isotopy classes of orientation-preserving homeomorphisms of 
$M$ that map each $X_i$ to itself, 
where the isotopy is required to preserve each $X_i$ as a set.  

Now, let $L$ be a link in a closed orientable $3$-manifold $M$.
For a bridge decomposition $(V_1, V_2, L; \Sigma)$, 
we call the mapping class group $\MCG(M, V_1, L)$ 
its \emph{Goeritz group}, and denote it by 
$\mathcal{G} (M, L ; \Sigma)$. 
It can be easily verified that the natural maps 
$\mathcal{G}(M, L; \Sigma) \to \MCG(\Sigma, \Sigma \cap L)$ and 
$\MCG(V_i, V_i \cap L) \to \MCG(\Sigma, \Sigma \cap L)$ obtained by 
restricting the representing maps to $\Sigma$ are injective. 
Thus, both the groups $\mathcal{G}(M, L; \Sigma)$ and $\MCG(V_i, V_i \cap L)$ $(i = 1, 2)$ can be regarded as subgroups of $\MCG(\Sigma, \Sigma \cap L)$, and under this identification, we can express the following relation:
\[ \mathcal{G}(M, L; \Sigma) = \MCG(V_1, V_1 \cap L) \cap \MCG(V_2, V_2 \cap L) < \MCG(\Sigma, \Sigma \cap L).\] 
In the discussion that follows, we actually use this subgroup hierarchy, often treating elements of $\mathcal{G}(M, L; \Sigma)$ 
as elements of $\MCG(V_1, V_1 \cap L)$ and $\MCG(\Sigma, \Sigma \cap L)$ as long as this does not cause confusion. 
We note that the Goeritz groups of equivalent bridge decompositions are clearly isomorphic by definition. 

%
%

\section{Mapping class groups of a trivial $1$-tangle in a solid torus}
\label{sec: Mapping class groups of a trivial 1-tangle in a solid torus}

Let $(V_1, V_2, K; \Sigma)$ be a $(1, 1)$-decomposition of a knot $K$ in a closed orientable $3$-manifold $M$. 
Since $(V_1, V_2; \Sigma)$ is a Heegaard splitting of genus $1$, $M$ is either $S^3$, $S^2 \times S^1$, or a lens space. 
From now, we use $K$ instead of $L$, because the links that allow $(1, 1)$-decomposition are knots. 
Let $T_i := V_i \cap K$ $(i = 1, 2)$ be the trivial $1$-tangle in $V_i$. 
In this section, we discuss the mapping class group $\MCG (V_1, T_1)$, which is of fundamental importance 
in finding a presentation of the Goeritz group $\mathcal{G} (M, L ; \Sigma)$.

We now define a simplicial complex, which will be used to obtain a finite presentation of $\MCG(V_1, T_1)$.
Identify $\Nbd (T_1; V_1)$ with $D^2 \times [0, 1]$, and let $C_{T_1}$ be the boundary of the disk $D^2 \times \{ \frac{1}{2} \}$. 
Let $\vtil$ be the exterior $E(T_1)$ of $T_1$, which is a handlebody of genus two. 
By attaching a $2$-handle $D^2 \times [0, 1]$ to $\vtil$ along $C_{T_1}$, we can recover $(V_1, T_1)$ from $(\vtil, C_{T_1})$.
Note that here we have a natural isomorphism  
$\MCG(V_1, T_1) \cong \MCG(\vtil, C_{T_1})$. 
In what follows, we often identify these two groups. 
A disk in $\vtil$ whose boundary intersects $C_{T_1}$ transversely at a single point is called a \emph{canceling disk} for $C_{T_1}$ (see Figure~\ref{figure:canceling_disks}).  
\begin{figure}[htbp]
\centering\includegraphics[width=14cm]{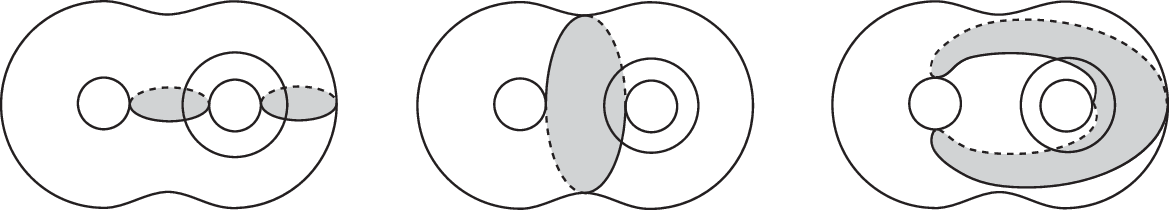}
\begin{picture}(400,0)(0,0)
\put(53,0){$\vtil$}
\put(75,21){$C_{T_1}$}

\put(195,0){$\vtil$}
\put(240,48){$C_{T_1}$}

\put(338,0){$\vtil$}
\put(381,48){$C_{T_1}$}

\end{picture}
\caption{Canceling disks of $C_{T_1}$.}
\label{figure:canceling_disks}
\end{figure}
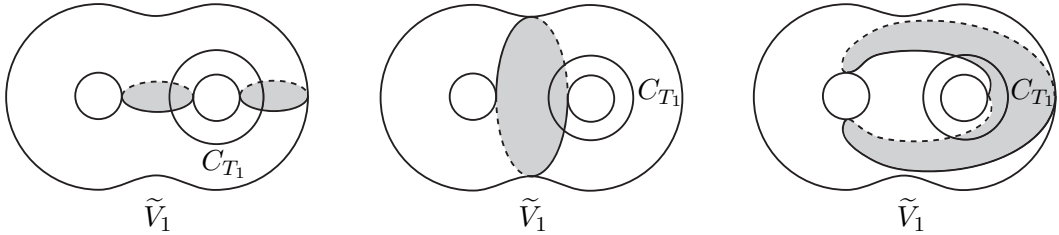
Clearly, there is a natural bijection between 
the set of isotopy classes of the canceling disks for $C_{T_1}$ in $\vtil$ 
and the set of isotopy classes of the canceling disks for $T_1$ in $V_1$.  

Let $V$ be a handlebody of genus $g$. 
The \emph{disk complex}, denoted by $\mathcal{D}(V)$, of $V$ is defined to be a simplicial complex 
whose vertex set is the set of isotopy classes of essential disks in $V$, where 
$k+1$ vertices span a $k$-simplex if they are distinct and they have pairwise disjoint representatives. 
The full subcomplex, denoted by $\mathcal{CD}(\vtil, C_{T_1})$, of the disk complex $\mathcal{D}(\vtil)$ spanned by 
vertices represented by canceling disks for $C_{T_1}$ is called a 
\emph{canceling disk complex}. 

In the following, we show that the complex $\mathcal{CD}(\vtil, C_{T_1})$ is a tree.

\begin{lemma}\label{canceling1}
The canceling disk complex $\mathcal{CD}(\vtil, C_{T_1})$ is $1$-dimensional. 
\end{lemma}
\begin{proof}
Suppose that there exists a simplex  of dimension $2$ or more in the complex $\mathcal{CD}(\vtil, C_{T_1})$.
Take three distinct vertices $D_0$, $D_1$, and $D_2$ from such a  simplex.
Since $D_0$, $D_1$, and $D_2$ intersect $C_{T_1}$ exactly once, they are all non-separating. 
Thus, the union $D_0 \cup D_1$ remains non-separating, 
because if we cut $\tilde{V}$ along $D_0$, 
the result is a solid torus, and $D_1$ then becomes its unique meridian disk, 
which is non-separating. 
It follows that $\vtil - (D_0 \cup D_1 \cup D_2)$ consists of at most two components.  
From this and the fact that the union $D_0 \cup D_1 \cup D_2$ intersects $C_{T_1}$ exactly three times, 
we see that the union $D_0 \cup D_1 \cup D_2$ is non-separating. 
This contradicts the fact that the genus of $\vtil$ is two.
\end{proof}

In order to prove the contractibility of the complex $\mathcal{CD}(\vtil, C_{T_1})$, 
we use the following sufficient condition given by Cho~\cite{Cho}.
First, we define an operation called the surgery of a disk, which is used to describe the condition.
Let $V$ be a handlebody of genus $g$. 
Let $D$ and $E$ be non-separating disks in $V$ with  $D \cap E \neq \emptyset$.
Let $\alpha$ be an outermost arc of $D \cap E$ in $D$,  
and $C$ a subdisk of $D$ with $C \cap E = \alpha$. 
Then, the arc $\alpha$ divides $E$ into two disks $E_1$ and $E_2$. 
Set $F_1 := E_1 \cup C$ and $F_2 := E_2 \cup C$.
The disks $F_1$ and $F_2$ are called \emph{disks obtained by surgery} on $E$ along $C$. 

\begin{theorem}[Cho~\cite{Cho}]\label{chokashuku}
Let $V$ be a handlebody of genus $g$.
Let $\mathcal{K}$ be a full subcomplex of the disk complex $\mathcal{D}(V)$.
If $\mathcal{K}$ satisfies the following condition, then $\mathcal{K}$ is contractible: 
\begin{itemize}
\item
For any disks $D$ and $E$ in $V$ that represent vertices of $\mathcal{K}$ and satisfy $D \cap E \neq \emptyset$, and for any outermost arc $\alpha$ of $D \cap E$ in $D$, at least one of the two disks $F_1$ and $F_2$ obtained by surgery on $E$ along a subdisk $C$ of $D$ with $C \cap E = \alpha$ represents a vertex of $\mathcal{K}$.
\end{itemize}
\end{theorem}

\begin{lemma} \label{cancelingkashuku}
The canceling disk complex $\mathcal{CD}(\vtil, C_{T_1})$ is contractible.
\end{lemma}
\begin{proof}
Since $\mathcal{CD}(\vtil, C_{T_1})$ is a full subcomplex of the disk complex $\mathcal{D} (\vtil)$, 
it suffices to show that $\mathcal{CD}(\vtil, C_{T_1})$ satisfies the condition 
in Theorem~\ref{chokashuku}. 
Let $D$ and $E$ be canceling disks of $C_{T_1}$ such that $D \cap E \neq \emptyset$.  
Let $\alpha$ be an outermost arc of $D \cap E$ in $D$, and $C$ a subdisk of $D$ with $C \cap E = \alpha$. 
The arc $\alpha$ divides $E$ into two disks $E_1$ and $E_2$. 
Set $F_1 := E_1 \cup C$ and $F_2 := E_2 \cup C$, which are disks obtained by surgery on $E$ along $C$. 
Since $D$ and $E$ are canceling disks of $C_{T_1}$, each of $\partial D$ and $\partial E$ intersects $C_{T_1}$ transversely at exactly one point. 
Thus, without loss of generality, we can assume that $E_1$ intersects $C_{T_1} $ at a point, and $E_2 \cap C_{T_1} = \emptyset$. 
Now, if $C \cap C_{T_1} = \emptyset$, then $F_1$ is a canceling disk of $(\vtil, C_{T_1})$, and 
if $C$ intersects $C_{T_1} $ at a point, then $F_2$ is a canceling disk of $(\vtil, C_{T_1})$. 
\end{proof}

From Lemmas~\ref{canceling1} and~\ref{cancelingkashuku}, we have the following. 

\begin{proposition}\label{tree}
The canceling disk complex $\mathcal{CD}(\vtil, C_{T_1})$ is a tree. 
\end{proposition}

Let $\mathcal{T}$ be the first barycentric subdivision of $\mathcal{CD}(\vtil, C_{T_1})$. 
This is a bipartite graph, with ``white" vertices of valence two represented
by pairs of canceling disks of $(\vtil, C_{T_1})$, and 
``black" vertices of countably infinite valence represented by a single canceling disk of $(\vtil, C_{T_1})$.
It is straightforward to see that the group $\MCG(\vtil, C_{T_1})$ acts on $\mathcal{T}$ simplicially, and 
the quotient $ \mathcal{T} / \MCG(\vtil, C_{T_1})$ is a single edge with two endpoints, 
one is represented by a ``white" vertex and the other is represented by a ``black" vertex, see Figure ~\ref{figure:quotientT}. 
Fix a pair $\{ D, E \}$ of distinct, disjoint canceling disks for $C_{T_1}$. 
\begin{figure}[htbp]
\centering\includegraphics[width=2.5cm]{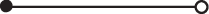}
\begin{picture}(400,0)(0,0)
\put(158,0){$[D]$}
\put(212,0){$[ \{ D , E \}]$}
\end{picture}
\caption{The quotient graph $\mathcal{T} / \MCG(\vtil, C_{T_1})$.}
\label{figure:quotientT}
\end{figure}
By Bass-Serre theory~\cite{Serre} for groups acting on trees, we have 
$\MCG(\vtil, C_{T_1}) = G_{[D]} *_{ G_{ ( [D], [E] )} } G_{ [ \{D, E \} ] }$, where  
\begin{itemize}
\item
$G_{[D]} := \MCG (\vtil , C_{T_1}, D) < \MCG (\vtil, C_{T_1})$; 
\item
$G_{ [ \{D, E\} ] } := \MCG (\vtil , C_{T_1}, D \cup E) < \MCG (\vtil, C_{T_1})$; and 
\item
$G_{ [( D, E )]} := \MCG (\vtil , C_{T_1}, D , E) < \MCG (\vtil, C_{T_1})$.
\end{itemize}

Let $m_1$ be the boundary of a non-separating disk of $\vtil$ disjoint from $C_{T_1}$. 
This is uniquely determined by Corollary~2.2 of Funayoshi--Koda~\cite{Funayoshi-Koda}. 
We define four specific elements $\alpha$, $\beta_D$, $\tau$, $\gamma_{\{D, E\}}$ of $\MCG(\vtil, C_{T_1})$ 
as follows. 

\begin{itemize}
\item 
The map $\alpha$ is an extension of the hyperelliptic involution of $\partial \vtil$. 
See the left top in Figure~\ref{figure:generators}.
\item 
Set $\ell_D := \partial (\Nbd (\partial D \cup C_{T_1}; \partial \vtil))$. 
The map $\beta_D$ is defined by a half-twist along $\ell_D$ as shown on the right top in Figure~\ref{figure:generators}. 
\item 
The map $\tau$ is an extension of the Dehn twist along $m_1$. See the left bottom in Figure~\ref{figure:generators}.
\item 
We denote the two simple closed curves 
$\partial (\Nbd ((\partial D \cup \partial E) \cup C_{T_1}; \partial \vtil))$ in $\partial \vtil$ 
by $\ell_{\{D, E\}}$ and $\ell'_{\{D, E\}}$. 
Then the map $\gamma_{\{D, E\}}$ is defined by a half-twist along $\ell_{\{D, E\}} \cup \ell'_{\{D, E\}}$ 
as shown on the right bottom in Figure~\ref{figure:generators}. 
\end{itemize}
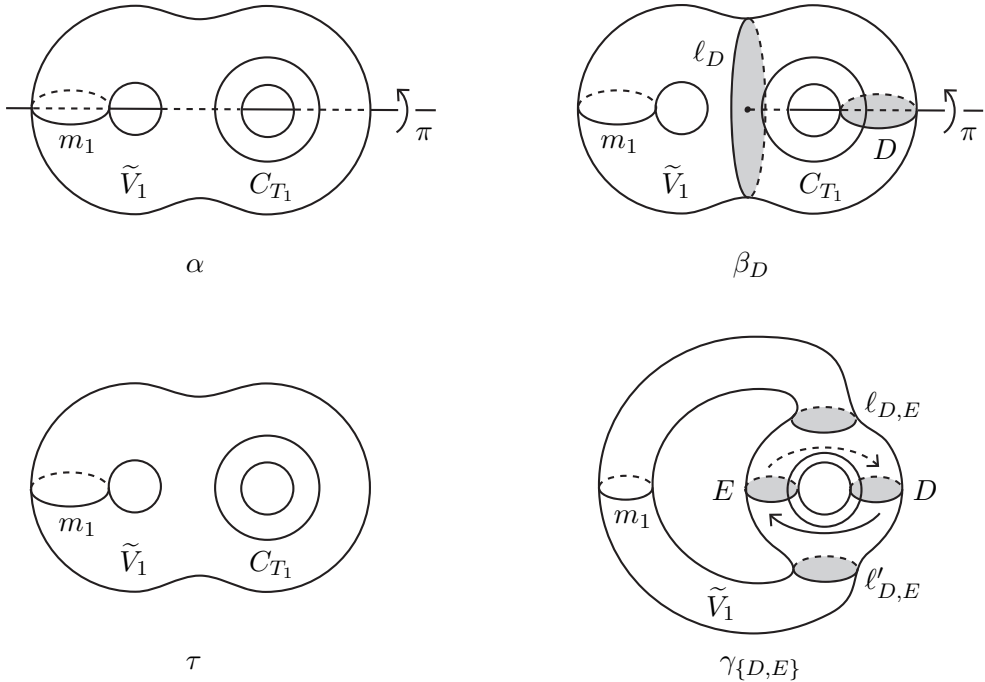
\begin{figure}[htbp]
\centering\includegraphics[width=13cm]{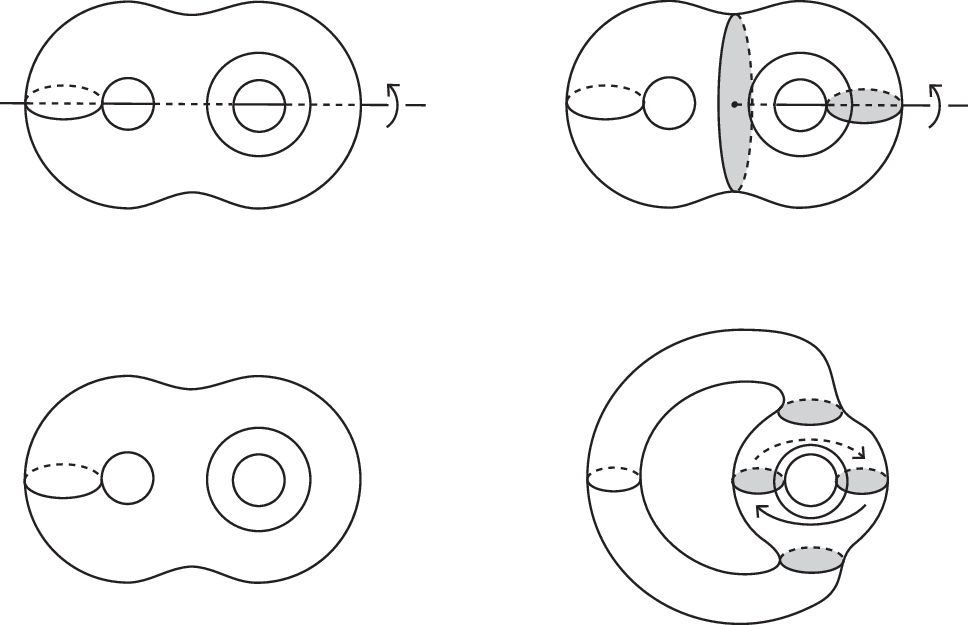}
\begin{picture}(400,0)(0,0)
\put(83,150){$\alpha$}
\put(170,200){$\pi$}
\put(107,180){$C_{T_1}$}
\put(35,197){$m_1$}
\put(58,180){$\vtil$}

\put(290,150){$\beta_D$}
\put(376,200){$\pi$}
\put(343,193){$D$}
\put(315,180){$C_{T_1}$}
\put(275,230){$\ell_D$}
\put(240,197){$m_1$}
\put(263,180){$\vtil$}

\put(83,0){$\tau$}
\put(107,37){$C_{T_1}$}
\put(35,52){$m_1$}
\put(58,37){$\vtil$}

\put(285,0){$\gamma_{\{ D, E \} }$}
\put(245,55){$m_1$}
\put(282,64){$E$}
\put(358,64){$D$}
\put(340,97){$\ell_{D,E}$}
\put(340,30){$\ell'_{D,E}$}
\put(280,20){$\vtil$}

\end{picture}
\caption{The four elements $\alpha, \beta_D, \tau, \gamma_{\{D, E\}} \in \MCG(\vtil, C_{T_1})$.}
\label{figure:generators}
\end{figure}
Note that the orders of $\alpha$ and $\gamma_{\{D, E\}}$ are both two, whereas 
 $\beta_D$ and $\tau$ have infinite order.

\begin{lemma}\label{G_D}
\begin{enumerate}
\item
$G_{[D]} = \ZZ / 2 \ZZ \langle \alpha \rangle \times \ZZ \langle \beta_D \rangle \times \ZZ \langle \tau \rangle$. 
\item
$G_{ [ \{D, E\} ] } = \ZZ / 2 \ZZ \langle \alpha \rangle \times \ZZ / 2 \ZZ \langle \gamma_{\{D, E\}} \rangle \times \ZZ \langle \tau \rangle$. 
\item
$G_{ [ ( D, E ) ] } = \ZZ / 2 \ZZ \langle \alpha \rangle \times \ZZ \langle \tau \rangle$.
\end{enumerate}
\end{lemma}
\begin{proof}
We set $\beta = \beta_D$, $\gamma = \gamma_{\{D, E\}}$ for simplicity. 
Note that every element of the mapping class group $\MCG (\vtil, C_{T_1})$ preserves $m_1$, for 
the simple closed curve $m_1$ is uniquely determined from $C_{T_1}$ up to isotopy. 

First, we show (1). 
Since the elements of $G_{[D]}$ preserve $C_{T_1}$ and $D$, they also preserve 
$\ell_D = \Nbd(\partial D \cup C_{T_1}; \partial \vtil)$.  
We cut $\vtil$ along the disk bounded by $\ell_D$ into two solid tori $X_1$ and $X_2$, where 
$X_2$ is the one containing $C_{T_1}$ as shown in Figure~\ref{figure:X_1_X_2}. 
Since the elements of $G_{[D]}$ preserve $C_{T_1}$, the solid tori $X_1$ and $X_2$ are not interchanged by elements of $G_{[D]}$.
\begin{figure}[htbp]
\centering\includegraphics[width=8.5cm]{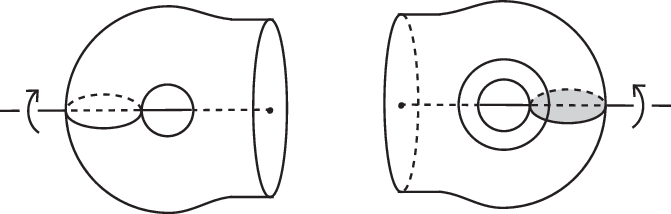}
\begin{picture}(400,0)(0,0)
\put(80,40){$\pi$}
\put(80,60){$\beta'$}
\put(132, 0){$X_1$}
\put(112, 35){$m_1$}
\put(157,70){$\ell_D$}

\put(313,40){$\pi$}
\put(314,60){$\beta$}
\put(252, 0){$X_2$}
\put(252, 25){$C_{T_1}$}
\put(278, 35){$D$}
\put(203,70){$\ell_D$}
\end{picture}
\caption{The decomposition of $\vtil$ into $X_1$ and $X_2$.}
\label{figure:X_1_X_2}
\end{figure}
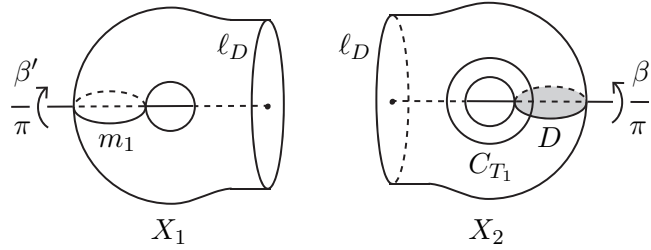
Note that the disk in $\vtil$ bounded by $m_1$ is the meridian disk of $X_1$. 
First, let $\varphi$ be an element of $G_{[D]}$ that fixes $X_2$. 
Since $\varphi$ preserves $m_1$, $\varphi$ can be written as a product of powers of maps $\beta'$ and $\tau$ shown in Figure~\ref{figure:X_1_X_2}. 
Next, let $\psi$ be an element of $G_{[D]}$ that fixes $X_1$. 
Since $\psi$ preserves both $D$ and $C_{T_1}$, we see that $\psi$ is just a power of  $\beta$. 
Here, we have $\alpha = \beta \beta'$.
Therefore, $G_{[D]}$ is generated by $\alpha$, $\beta$, and $\tau$. 
The elements $\alpha$, $\beta$ and $\tau$ are mutually commutative, and we also have $\alpha^2 = 1$.
We show that there are no other essential relations for $G_{[D]}$.

Choose an arbitrary product $\alpha^l \beta^m \tau^n$ that is the identity in $G_{[D]}$. 
By taking its square, we have $(\beta^2)^{m} \tau^{2n} = 1$.  
Here, $\beta^2$ is nothing but (the extension of the) Dehn twist along the separating simple closed curve $\ell_{D} \subset \partial \vtil$. 
Thus, the isomorphism $(\beta^2)_{*}$ of $H_1(\partial \vtil; \ZZ)$ induced by $\beta^2$ is the identity by 
Farb--Margalit~\cite[Proposition~6.3]{FM12}. 
It follows that the isomorphism $(\tau^{2n})_{*}$ of $H_1(\partial \vtil; \ZZ)$ induced by $\tau^{2n}$ is also the identity, for $(\beta^2)^{m} \tau^{2n} = 1$. 
Then, again by Farb--Margalit~\cite[Proposition~6.3]{FM12}, we have $n=0$. 
From this, we see $m= 0$ and then $l = 0$. 
Consequently, we have 
\begin{align*}
G_{[D]} &= 
\langle 
\alpha, \beta, \tau \mid 
\alpha \beta = \beta \alpha, 
\alpha \tau = \tau \alpha, 
\beta \tau = \tau \beta,
\alpha^2 =1
\rangle\\
&\cong \ZZ/2\ZZ\langle \alpha \rangle \times \ZZ\langle \beta \rangle \times \ZZ \langle \tau \rangle
\end{align*}

Next, we show (2). 
By the proof of (1) and the fact that $\beta^N$ does not preserve $E$ for any nonzero integer $N$, 
we see that the maps preserving both $D$ and $E$ are generated by $\alpha$ and $\tau$. 
Since the map $\gamma$ interchanges $D$ and $E$, $G_{[ \{ D, E \}]}$ is generated by $\alpha$, $\tau$, and $\gamma$.
The elements $\alpha$, $\tau$ and $\gamma$ are mutually commutative, and we have $\alpha^2 = \gamma^2 = 1$. 
We show that there are no other essential relations for $G_{[\{D, E\}]}$. 

Choose an arbitrary product $\alpha^l \gamma^m \tau^n$ that is the identity in $G_{[\{D, E\}]}$. 
By taking its square, we have $\tau^{2n} = 1$. 
Then, since  the order of $\tau$ is infinite by  Farb--Margalit~\cite[Proposition~3.2]{FM12}, we have $n = 0$. 
Consequently, we have 
\begin{align*}
G_{[\{D, E\}]} &= 
\langle 
\alpha, \tau, \gamma \mid 
\alpha \tau = \tau \alpha, 
\alpha \gamma = \gamma \alpha, 
\tau \gamma = \gamma \tau, 
\alpha^2 = 1, 
\gamma^2 =1
\rangle\\
&\cong \ZZ / 2 \ZZ \langle \alpha \rangle \times \ZZ / 2 \ZZ \langle \gamma \rangle \times \ZZ \langle \tau \rangle. 
\end{align*}

Finally, we show (3).
Since the elements of $G_{ [ ( D , E ) ] }$ preserve both $D$ and $E$, 
$G_{ [ ( D, E ) ] }$ is generated by $\alpha$ and $\tau$, as already mentioned in the proof of (2).
Moreover, the relations are only $\alpha \tau = \tau \alpha$ and $\alpha^2 = 1$ essentially. 
Therefore, we have 
\begin{align*}
G_{ [ ( D, E ) ] } &=
\langle 
\alpha, \tau \mid 
\alpha \tau = \tau \alpha, 
\alpha^2 = 1
\rangle
\cong \ZZ / 2 \ZZ \langle \alpha \rangle \times \ZZ \langle \tau \rangle. 
\end{align*}
\end{proof}

\begin{proposition}\label{MCGV_1}
We have $\MCG(V_1, T_1) = \ZZ/2\ZZ \langle \alpha \rangle \times \ZZ \langle \tau \rangle \times \langle \beta, \gamma \mid \gamma^2 = 1 \rangle$.
\end{proposition}
\begin{proof}
Set $\beta = \beta_D$ and $\gamma = \gamma_{\{D, E\}}$. 
By Lemma~\ref{G_D}, we have 
\begin{align*}
\MCG(V_1, T_1) 
&\cong \MCG(\vtil, C_{T_1}) = G_{[D]} \ast_{G_{ [ ( D , E ) ] }} G_{[\{ D, E \}]}\\
&= \left\langle
\alpha, \beta, \tau, \gamma ~\left| ~
\begin{array}{l}
\alpha \beta = \beta \alpha,
\alpha \tau = \tau \alpha,
\alpha \gamma = \gamma \alpha, \\
\beta \tau = \tau \beta,
\tau \gamma = \gamma \tau,\alpha^2 =1, \gamma^2 = 1
\end{array} \right. \right\rangle\\
&\cong \ZZ/2\ZZ \langle \alpha \rangle \times \ZZ \langle \tau \rangle \times \langle \beta, \gamma \mid \gamma^2 = 1 \rangle .
\end{align*}
\end{proof}

\section{Proof of Main Theorem}
\label{sec: Proof of Main Theorem}

We first give the definition of the $(1, 1)$-decomposition $(M, K_{p/q}; \Sigma_{p/q})$ mentioned in Theorem \ref{shukekka}. 

\begin{definition}\label{Kpq}
Let $p$ and $q$ be coprime integers, where $p$ is a non-zero even number.
For $i = 1, 2$, let $V_i$ be the solid torus and $T_i$ be the trivial $1$-tangle in $V_i$. 
Let $m \subset \partial V_1$ be the meridian of $V_1$ containing $\partial T_1 = \{ t, t' \}$. 
Let $f : \RR^2 \to \partial V_1( = \RR^2 / \ZZ^2) $ be the universal covering map, where 
we assume that $f^{-1} (m) = \RR \times \ZZ$,  
$f^{-1}( t ) = \ZZ^2$, $f^{-1}( t' ) = \{ n+\frac{1}{2} \mid n \in \ZZ \} \times \ZZ$. 
Let $L_{p/q}$ be the line in $\RR^2$ defined by $y= \frac{p}{q} x  + \sqrt{2}$. 
Then the image $f \left( L_{p/q} \right)$ is a simple closed curve in $\partial V_1 - \{ t, t' \}$, see Figure ~\ref{figure:Kpqdef}. 
\begin{figure}[htbp]
\centering\includegraphics[width=10cm]{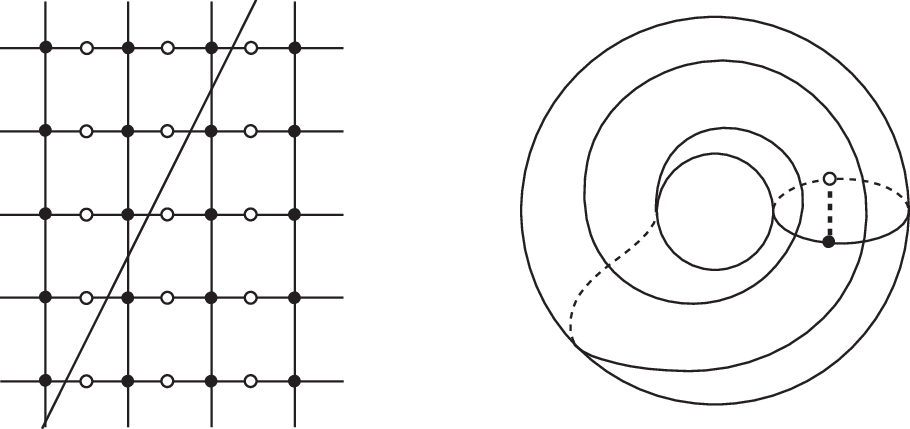}
\begin{picture}(400,0)(0,0)

\put(101,65){$L_{p/q}$}
\put(105,0){$\RR^2$}

\put(185,77){$\xrightarrow{f}$}

\put(345,79){$m$}
\put(197,33){$f(L_{p/q})$}
\put(314,60){$t$}
\put(314,96){$t'$}
\put(275,5){$V_1$}
\end{picture}
\caption{The line $ L_{p/q} \subset \RR^2$ and the simple closed curve $f \left( L_{p/q} \right) \subset \partial V_1- \{t, t' \}$.}
\label{figure:Kpqdef}
\end{figure}
Let $m_2$ be the boundary of a non-separating disk of $V_2$ disjoint from $T_2$. 
Then, by attaching $\partial V_2$ to $\partial V_1$ so that 
$m_2$ and $f \left( L_{p/q} \right)$ coincide and  
the endpoints of $T_1$ and $T_2$ meet each other, 
we get a $(1,1)$-decomposition of a certain knot $K_{p/q}$ in the lens space $L(p, q)$. 
By setting $\Sigma_{p/q} := \partial V_1 = \partial V_2 \subset L(p, q)$, we can write the 
$(1,1)$-decomposition thus obtained by $(L(p, q), K_{p/q}; \Sigma_{p/q})$. 
\end{definition}

We shall prove several lemmas that will lead to a proof of the main theorem. 
In the following arguments, for a $(1,1)$-decomposition $(V_1, V_2, K; \Sigma)$ of a knot $K$ in a closed orientable $3$-manifold $M$, 
we use the following symbols. 
\begin{itemize}
\item
$T_i := V_i \cap L$ ($i=1 , 2$). 
\item
$\widetilde{V}_i$ ($i=1 , 2$) denotes the exteriof of $T_i \subset V_i$, which is a handlebody of genus two. 
\item
$m_i$ ($i=1,2$) denotes the boundary of the unique non-separating disk in $\widetilde{V}_i$ disjoint from $C_{T_i}$. 
This is sometimes identified with the (unique) meridian disk of the solid torus $V_i$ disjoint from $T_i$ in a natural way. 
\end{itemize}

\begin{lemma}\label{m1m2lem}
Let $(M, K; \Sigma)$ be a $(1, 1)$-decomposition of a knot $K$ in a closed orientable $3$-manifold $M$. 
If $m_1 \cap m_2 \neq \emptyset$, then 
there exist at most two canceling disks for $(\vtil, C_{T_1})$ in $\vtil$ 
that intersect $m_2$ minimally among all canceling disks for $(\vtil, C_{T_1})$ in $\vtil$.  
\end{lemma}
\begin{proof}
Let $\Sigma'$ be a torus with two holes obtained by cutting $\partial \vtil$ along $m_1$. 
We denote by $m_1^{\pm}$ the two boundary components of $\Sigma'$. 
Then $\Sigma' \cap m_2$ consists of essential arcs in $\Sigma'$ as shown on the left in 
Figure~\ref{figure:sigma04}. 
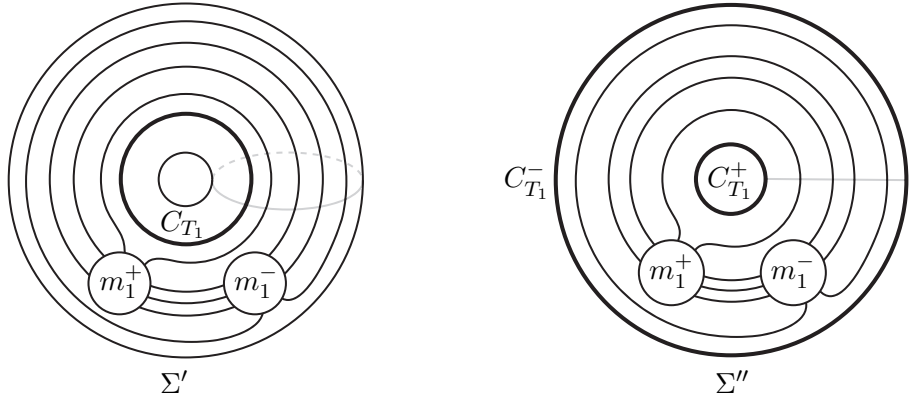
\begin{figure}[htbp]
\centering\includegraphics[width=12cm]{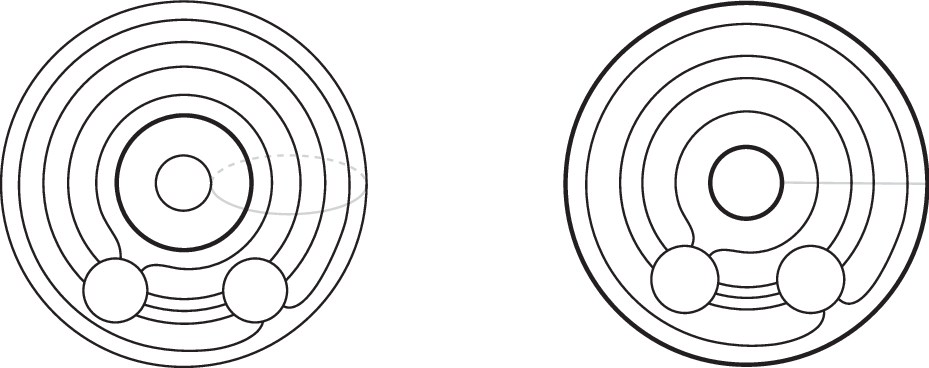}
\begin{picture}(400,0)(0,0)
\put(64,39){$m_1^+$}
\put(115,39){$m_1^-$}
\put(87,61.5){$C_{T_1}$}
\put(87,0){$\Sigma'$}

\put(294,78){$C_{T_1}^+$}
\put(217,78){$C_{T_1}^-$}
\put(272,43){$m_1^+$}
\put(318,43){$m_1^-$}
\put(297,0){$\Sigma''$}
\end{picture}
\caption{The two-holed torus $\Sigma'$ (left) and the four-holed sphere $\Sigma''$ (right).}
\label{figure:sigma04}
\end{figure}
We further cut $\Sigma'$ along $C_{T_1}$ to obtain a sphere $\Sigma''$ with four holes. 
The two new boundary components of $\Sigma''$ are denoted by $C_{T_1}^\pm$, 
see the right in Figure~\ref{figure:sigma04}. 
Now, $\Sigma'' \cap \partial D$ consists of arcs connecting $C_{T_1}^{+}$ and $C_{T_1}^{-}$. 
From this, we see that there exist at most two canceling disks for $C_{T_1}$ in $\vtil$ that intersect $m_2$ minimally.  
\end{proof}

\begin{lemma}\label{uniqueE_1}
Let $(M, K; \Sigma)$ be a $(1, 1)$-decomposition of a knot $K$ in a closed orientable $3$-manifold $M$. 
Let $m_1 \cap m_2 \neq \emptyset$. 
If there exists a unique canceling disk for $C_{T_1}$ in $\vtil$ having minimal intersection with $m_2$, 
then the canceling disk for $C_{T_1}$ in $\vtil$ having the second least intersection with $m_2$ is also unique. 
\end{lemma}

\begin{proof}
This follows from the arguments of Lemma~\ref{m1m2lem} and Figure~\ref{figure:sigma04}. 
\end{proof}

\begin{lemma}\label{preservem2}
Let $(M, K; \Sigma)$ be a $(1, 1)$-decomposition of a knot $K$ in a closed orientable $3$-manifold $M$. 
Then, for an element $\varphi$ of $\MCG(V_1, T_1)$, 
we have $\varphi (m_2) = m_2$ if and only if $\varphi \in \mathcal{G}(M, K; \Sigma)$. 
\end{lemma}

\begin{proof}
From the uniqueness of $m_2$, we have 
$\varphi (m_2) = m_2$ for any $\varphi \in \mathcal{G}(M, K; \Sigma)$. 
Suppose that $\varphi \in \MCG(V_1, T_1)$ satisfies $\varphi (m_2) = m_2$.  
Let $D_{m_2}$ be a disk in $V_2$ bounded by $m_2$ and disjoint from $T_2$. 
By assumption, $\varphi$ can be extended to a homomorphism of $V_1 \cup \Nbd(D_{m_2})$. 
Since $M - (V_1 \cup \Nbd(D_{m_2}))$ is a $3$-ball, we can extend $\varphi$ further to a homomorphism of 
all of $M$. 
Here, we can assume that this extended $\varphi$ preserves $T_2$ by the uniqueness of the trivial tangle up to isotopy, 
which implies $\varphi \in \mathcal{G}(M, K; \Sigma)$.
\end{proof}

By virtue of Lemma~\ref{preservem2}, in order to check whether an element of $\MCG(V_1, T_1)$ 
is contained in the Goeritz group $\mathcal{G}(M, K; \Sigma)$, we only need to check whether that map preserves $m_2$.
In other words, we have the following: 
\[\mathcal{G}(M, K; \Sigma) = \MCG(V_1, T_1, m_2) < \MCG(V_1, T_1).\]

Given a fixed pair $\{D, E\}$ of disjoint, non-parallel canceling disks for $C_{T_1}$ in $\vtil$, 
let $\alpha$, $\beta_D$, $\tau$, $\gamma_{\{D, E\}}$ in $\MCG(\vtil, C_{T_1})$ (or $\MCG(V_1, T_1)$) 
be the maps defined in Section \ref{sec: Mapping class groups of a trivial 1-tangle in a solid torus}. 
Note that $\alpha$ and $\tau$ are determined independently of the choice of $\{ D, E \}$. 
Moreover, the map $\alpha$, which is the hyperelliptic involution of $\partial \vtil$, does not even depend on the 
$(1,1)$-decompositions $(M, K; \Sigma)$. 
Since the hyperbolic involution $ \alpha $ 
is a central element of the mapping class group of the closed orientable surface of genus $2$, 
it preserves every simple closed curve on the surface (see~\cite[Section 3.4]{FM12}). 
Therefore,  
we always have $\alpha \in \mathcal{G}(M, K; \Sigma)$ by Lemma ~\ref{preservem2}. 

Let $(L(p, q), K_{p/q}; \Sigma_{p/q})$ be the $(1,1)$-decomposition introduced at the beginning of this section, where we recall that 
$p$ is assumed to be a non-zero even number. 
By definition, there exist canceling disks $D$, $E$ for $T_1$ in $V_1$ such that 
$\gamma_{\{D, E\}} \in \mathcal{G}(L(p, q), K_{p/q}; \Sigma_{p/q})$. 
In fact, let $f$ be the universal covering map as in that definition, and 
let $D$ and $E$ be disks in $V_1$ bounded by $T_1 \cup f([0, \frac{1}{2}] \times \{0\})$ and $T_1 \cup f([\frac{1}{2}, 1] \times \{0\})$, respectively. 
Then, since $p$ is even, we have $\gamma_{\{D, E\}}(m_2) = m_2$, as illustrated in Figure~\ref{figure:Kpqgamma}. 
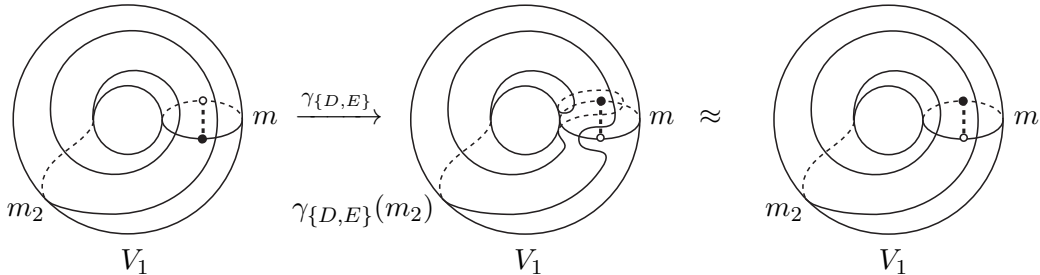
\begin{figure}[htbp]
\centering\includegraphics[width=13.2cm]{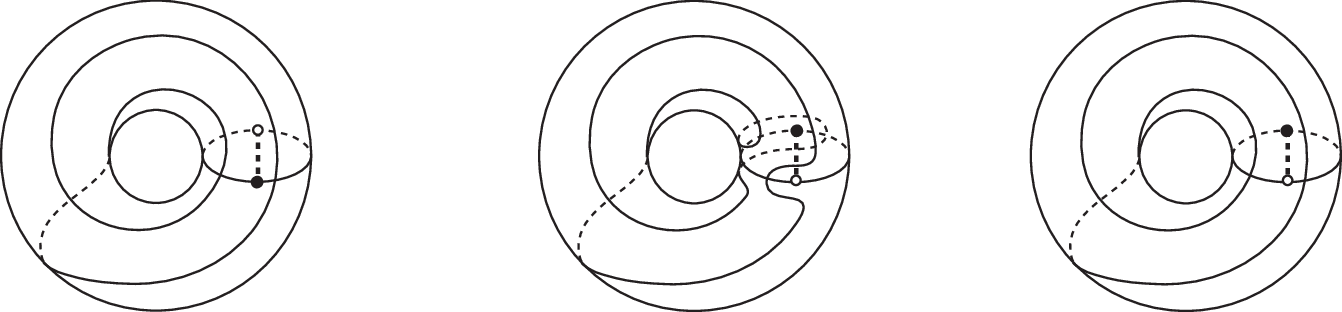}
\begin{picture}(400,0)(0,0)
\put(103,55){$m$}
\put(10,20){$m_2$}
\put(53,0){$V_1$}

\put(119,55){$\xrightarrow{\gamma_{\{D,E\}}}$}

\put(253,55){$m$}
\put(118,20){$\gamma_{\{ D,E \}}(m_2)$}
\put(202,0){$V_1$}

\put(272,55){$\approx$}

\put(391,55){$m$}
\put(297,20){$m_2$}
\put(340,0){$V_1$}

\end{picture}
\caption{The map $\gamma_{\{D,E\}}$ for $(L(p, q), K_{p/q}; \Sigma_{p/q})$.}
\label{figure:Kpqgamma}
\end{figure}
In fact, $\gamma_{\{D, E\}}$ lifts to the Euclidean transformation $(x, y) \mapsto (x + 1/2 , y)$ of $\RR^2$, which preserves 
$f^{-1} (f (L_{p/q}))$ as a set. 
Thus, by Lemma~\ref{preservem2}, we have $\gamma_{\{D, E\}} \in \mathcal{G}(L(p, q), K_{p/q}; \Sigma_{p/q})$. 
Here we note that $D$ and $E$ are actually 
 the two canceling disks for $T_1$ in $V_1$ that intersect $m_2$ minimally.

\begin{lemma}\label{coredouti}
Let $(M, K; \Sigma)$ be a $(1, 1)$-decomposition of a knot $K$ in a closed orientable $3$-manifold $M$. 
Then, we have both $m_1 \cap m_2 = \emptyset$ and $m_1 \neq m_2$ if and only if 
$M = S^2 \times S^1$ and $K$ is the core knot. 
\end{lemma}

\begin{proof}
By Theorem~\ref{kyoridouti} (2), it suffices to show that 
we have both $m_1 \cap m_2 = \emptyset$ and $m_1 \neq m_2$ if and only if $d(M, K; \Sigma) = 1$. 

Suppose that $m_1 \cap m_2 = \emptyset$ and $m_1 \neq m_2$. 
Since $m_1$ and $m_2$ are disjoint, the ambient $3$-manifold $M$ is $S^2 \times S^1$. 
Further, since $m_i \in \mathcal{D}(V_i - T_i)$, we have $d(M, K; \Sigma) \leq 1$. 
If $d(M, K; \Sigma) = 0$, $K$ is the trivial knot by Theorem~\ref{kyoridouti} (1). 
This is, however, impossible because if $K$ is the trivial knot in $S^2 \times S^1$, its unique $(1,1)$-decomposition is 
the one obtained by the double of the trivial $1$-tangle in the solid torus, which implies that 
$m_1 = m_2$. 

Suppose $d(M, K; \Sigma) = 1$. 
Then there exists a simple closed curve 
$c_i \in \mathcal{D}(V_i - T_i)$ for each $i = 1, 2$ such that $c_1 \neq c_2$ and $c_1 \cap c_2 = \emptyset$. 
We show that $c_1=m_1$ and $c_2=m_2$. 
We assume, for a contradiction, that $c_2 \neq m_2$. 
Recall that $m_2$ is the boundary of a 
unique non-separating disk in $V_2$ disjoint from $T_2$. 
Thus, $c_2$ is separating. 
It follows that $c_1$ is non-separating, otherwise we have $c_1 = c_2$, a contradiction. 
Then, by  Saito~\cite[Lemma 3.3]{Saito}, $c_2$ also bounds a (separating) disk in $V_1$ disjoint from $T_1$. 
This implies that $d(M, K; \Sigma) = 0$, which is a contradiction. 
Consequently, we have $c_2 = m_2$. 
The proof of $c_1=m_1$ is symmetric to the one above. 
\end{proof}

For oriented simple closed curves $c_1$, $c_2$ in a closed oriented surface $\Sigma$, 
we denote by $\hat{\iota}(c_1, c_2)$ and $\iota (c_1, c_2)$ 
the algebraic and geometric intersection numbers, respectively. 
Recall that the latter is defined by  
\[ \iota (c_1, c_2) = \min \{ | c_1' \cap c_2' | \mid \mbox{$c_i'$ is isotopic to $c_i$} \}, \]
and thus, it does not depend on the orientations. 
By Farb--Margalit~\cite[Proposition~3.2]{FM12}, we have the equality 
$\iota(c_2, \tau_{c_1}^{\,k}(c_2)) = |k| \iota(c_1, c_2)^2$, 
where $\tau_{c_1}$ denotes the Dehn twist along $c_1$. 
Thus we have, in particular, the following. 

\begin{lemma}\label{kikakoutensu}
Let $n \in \ZZ - \{ 0 \}$. 
For a simple closed curve $c$ in $\partial \vtil$ with $c \cap m_1 \neq \emptyset$, 
we have $\iota(c, \tau^n(c)) \neq 0$, thus, $ c \neq \tau^n(c) $. 
\end{lemma}
The point here is the so-called \emph{bigon criterion}, which asserts that 
two simple closed curves in $\Sigma$ intersects minimally if and only if they do not form a bigon. 
Using this criterion, we can also prove the following lemma with nearly identical arguments. 

\begin{lemma}\label{kikakoutensu-half-twist}
Let $n \in \ZZ - \{ 0 \}$. 
For a simple closed curve $c$ in $\partial \vtil$ with $c \cap \ell_D \neq \emptyset$, 
we have $\iota(c, \beta_D^n(c)) \neq 0$, thus, $c \neq \beta_D^n(c)$. 
\end{lemma}

We are now going to see that the Goeritz group 
$\mathcal{G}(M, K; \Sigma)$ is determined according to the following cases. 

\begin{description}
\item[\textrm{Case 1}] 
$m_1 \cap m_2 \neq \emptyset$. 
\begin{enumerate}
\renewcommand{\labelenumi}{(\alph{enumi})}
\item
There exists a unique canceling disk of $C_{T_1}$ whose intersection with $m_2$ is minimal (denote this disk as $D_1$). 
\begin{enumerate}
\renewcommand{\labelenumii}{(\roman{enumii})}
\item
$D_1 \cap m_2 = \emptyset$. 
\item
$D_1 \cap m_2 \neq \emptyset$. 
\end{enumerate}

\item
There exist exactly two canceling disks of $C_{T_1}$ whose intersection with $m_2$ is minimal  (denote these disks as $D_1$ and $E_1$). 
\begin{enumerate}
\renewcommand{\labelenumii}{(\roman{enumii})}
\item
$\iota(\partial(D_1 \cup E_1), m_2) = |\hat{\iota}(\partial(D_1 \cup E_1), m_2)|$.
\item
$\iota(\partial(D_1 \cup E_1), m_2) \neq |\hat{\iota}(\partial(D_1 \cup E_1), m_2)|$.
\end{enumerate}
\end{enumerate}

\item[\textrm{Case 2}] 
$m_1 \cap m_2 = \emptyset$.
\begin{enumerate}
\renewcommand{\labelenumi}{(\alph{enumi})}
\item
$m_1 = m_2$.
\item
$m_1 \neq m_2$.
\end{enumerate}
\end{description}

In the following, we set $\beta = \beta_{D_1}$, $\gamma = \gamma_{\{D_1, E_1\}}$ for simplicity. 

First, we consider Case 1-(a). 
By Lemma ~\ref{G_D} (1), we have 
$\mathcal{G}(M, K; \Sigma) < G_{[D_1]} = \ZZ /2 \ZZ \langle \alpha \rangle \times \ZZ \langle \beta \rangle \times \ZZ \langle \tau \rangle$. 

\noindent \textit{Case} 1-(a)-(i): 
Since we always have $C_{T_1} \cap m_2 = \emptyset$ from the construction, we have $\ell_{D_1} \cap m_2 = \emptyset$ in this case. 
Thus, $\beta$ preserves $m_2$. 
By Lemma~\ref{preservem2}, this implies that $\beta \in \mathcal{G}(M, K; \Sigma)$.
Since $m_1 \cap m_2 \neq \emptyset$, we have $\tau^n(m_2) \neq m_2$ for any $n \neq 0$ by Lemma~\ref{kikakoutensu}.
This implies $\tau^n \notin \mathcal{G}(M, K; \Sigma)$ for any $n \neq 0$ again by Lemma~\ref{preservem2}. 
Therefore, we have $\mathcal{G}(M, K; \Sigma) = \ZZ / 2 \ZZ \langle \alpha \rangle \times \ZZ \langle \beta \rangle$. 

\noindent \textit{Case} 1-(a)-(ii): 
Suppose that we have $\beta^n \tau^m \in \mathcal{G}(M, K; \Sigma)$ for some $n, m \in \ZZ$. 
Let $E_1$ be the canceling disk of $C_{T_1}$ with the second least intersection with $m_2$, which is 
uniquely determined by Lemma~\ref{uniqueE_1}. 
Then, $\beta^n \tau^m$ preserves $E_1$. 
It is easy to see that any canceling disk of $C_{T_1}$ does not intersect $m_1$, thus, 
$\tau$ preserves $E_1$. 
Therefore, we have $E_1 = \beta^n \tau^m (E) = \beta^n(E_1)$,  which implies $n = 0$ by Lemma \ref{kikakoutensu-half-twist} because 
$\ell_D \cap \partial E_1 \neq \emptyset$.  
Now, we have $\tau^m \in \mathcal{G}(M, K; \Sigma)$ and thus, $\tau^m ( m_2 ) = m_2 $ by Lemma~\ref{preservem2}. 
Since $m_1 \cap m_2 \neq \emptyset$, this shows $m=0$ by Lemma~\ref{kikakoutensu}. 
In consequence, we have $\mathcal{G}(M, K; \Sigma) = \ZZ / 2 \ZZ \langle \alpha \rangle$.

Next, we consider Case 1-(b). 
By Lemma~\ref{G_D} (2), we have 
$\mathcal{G}(M, K; \Sigma) < G_{[\{ D_1, E_1 \}]} 
= \ZZ / 2 \ZZ \langle \alpha \rangle \times \ZZ \langle \gamma \rangle \times \ZZ \langle \tau \rangle$. 
Suppose that we have $\tau^{n} \gamma^m \in \mathcal{G}(M, K; \Sigma)$ for some $n \in \ZZ, m \in \{0, 1\}$. 
By taking its square, we have $\tau^{2n} \in \mathcal{G}(M, K; \Sigma)$. 
From Lemma~\ref{preservem2}, it follows that $\tau^{2n}(m_2) = m_2$, which implies $n=0$ by Lemma~\ref{kikakoutensu}. 

\noindent \textit{Case} 1-(b)-(i): 
In this case, we have $(M, K; \Sigma) = (L(p, q), K_{p/q}, \Sigma_{p/q})$ for some coprime integers $p$ and $q$, where 
$m_2$ is actually the image $f \left( L_{p/q} \right)$ given in the definition of $K_{p/q}$. 

Thus, $\gamma$ is an element of  $\mathcal{G}(M, K; \Sigma)$ and we have 
$\mathcal{G}(M, K; \Sigma) = \ZZ / 2 \ZZ \langle \alpha \rangle \times \ZZ / 2 \ZZ \langle \gamma \rangle$.

\noindent \textit{Case} 1-(b)-(ii): 
We suppose, for a contradiction, that $\gamma \in \mathcal{G}(M, K; \Sigma)$. 
Let $t$ and $t'$ be the endpoints of the tangle $T_1$.
Note that $\gamma(t) = t'$ and $\gamma(t') = t$.
Let $A$ be the annulus obtained by cutting $\partial V_1$ along  $\partial(D_1 \cup E_1)$, 
where $D_1 \cap D_2 = \partial D_1 \cap \partial D_2 = T_1$ and $D_1 \cup E_1$ is a meridian disk of $V_1$. 
Denote the two boundary components of $A$ by $b^+$ and $b^-$. 
See Figure~\ref{figure:Agamma}. 
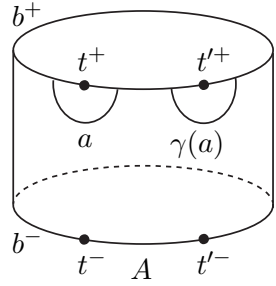
\begin{figure}[htbp]
\centering\includegraphics[width=3.5cm]{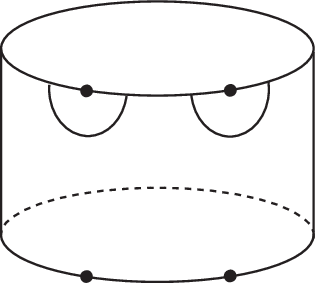}
\begin{picture}(400,0)(0,0)
\put(195,0){$A$}

\put(150,96){$b^+$}
\put(175,79){$t^+$}
\put(220,79){$t'^+$}
\put(175,50){$a$}
\put(210,48){$\gamma (a)$}

\put(150,10){$b^-$}
\put(175,2){$t^-$}
\put(220,2){$t'^-$}
\end{picture}
\caption{The annulus $A$ and its boundary circles $b^+$, $b^-$.}
\label{figure:Agamma}
\end{figure}

Let $t^+$ and $t'^+$ (resp. $t^-$ and $t'^-$) be the copies of $t$ and $t'$ on the circle $b^+$ (resp. $b^-$), respectively. 
Let $\mathcal{A}$ be the set of arcs on $A$ obtained by cutting $m_2$ along $\partial(D_1 \cup E_1)$.
Then, $\mathcal{A}$ is divided into three subsets: 
the set $\mathcal{A}_{+-}$ of arcs connecting $b^+$ and $b^-$, 
the set $\mathcal{A}_{++}$ of arcs connecting $b^+$ and $b^+$, 
and the set $\mathcal{A}_{--}$ of arcs connecting $b^-$ and $b^-$, where 
we have $ | \mathcal{A}_{++} | = | \mathcal{A}_{--} | $, because 
$| b^+ \cap m_2 | = | b^- \cap m_2 | = | \partial (D_1 \cup D_2)  \cap m_2 |$.  
By the assumption that $\iota(\partial(D_1 \cup E_1), m_2) \neq |\hat{\iota}(\partial(D_1 \cup E_1), m_2)|$, 
we have $| \mathcal{A}_{++} | \neq \emptyset$ (and thus, $| \mathcal{A}_{--} | \neq \emptyset$ as well). 
Without loss of generality, we can assume that 
there exists the innermost arc $a \in \mathcal{A}_{++}$ with respect to $t^+$. 
Since $\gamma \in \mathcal{G}(M, K; \Sigma)$, we have $\gamma(m_2) = m_2$, and thus, 
$\gamma(a)$ is also contained in $\mathcal{A}$. 
By definition, $\gamma$ maps $t$ to $t'$ without changing the orientation of the circle 
$\partial (D_1 \cup E_1)$. 
Therefore, $\gamma(a)$ is the innermost arc with respect to $t'^+$.
This contradicts the assumption that $m_2$ is a simple closed curve in $\partial V_1$. 
Therefore, $\gamma$ is not an element of the Goeritz group $\mathcal{G}(M, K; \Sigma)$. 
Consequently, we have $\mathcal{G}(M, K; \Sigma) = \ZZ / 2 \ZZ \langle \alpha \rangle$.

Finally we consider Case 2. 

\noindent \text{Case} 2-(a): 
In this case, we see that $\mathcal{G}(M, K; \Sigma) = \MCG(V_1, T_1, m_2) = \MCG(V_1, T_1)$, 
see Figure~\ref{figure:trivial_knot}.
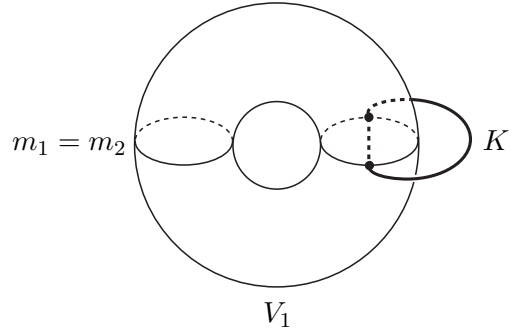
\begin{figure}[htbp]
\centering\includegraphics[width=4.5cm]{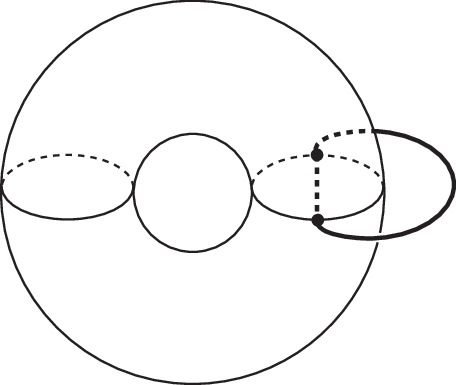}
\begin{picture}(400,0)(0,0)
\put(268,65){$K$}
\put(185,0){$V_1$}
\put(90,65){$m_1 = m_2$}

\end{picture}
\caption{The case where $m_1 \cap m_2 = \emptyset$ and $m_1 = m_2$.}
\label{figure:trivial_knot}
\end{figure}
Thus, we have $\mathcal{G}(M, K; \Sigma) = \ZZ / 2 \ZZ \langle \alpha \rangle \times \ZZ \langle \tau \rangle \times \langle \beta, \gamma \mid \gamma^2 = 1 \rangle$.

\noindent \text{Case} 2-(b): 
In this case, by Lemma 3.3 of Saito~\cite{Saito}, $m_2$ bounds a disk in $V_1$ that intersects $T_1$ transversely at a single point. 
See Figure~\ref{figure:core_knot}.
\begin{figure}[htbp]
\centering\includegraphics[width=11cm]{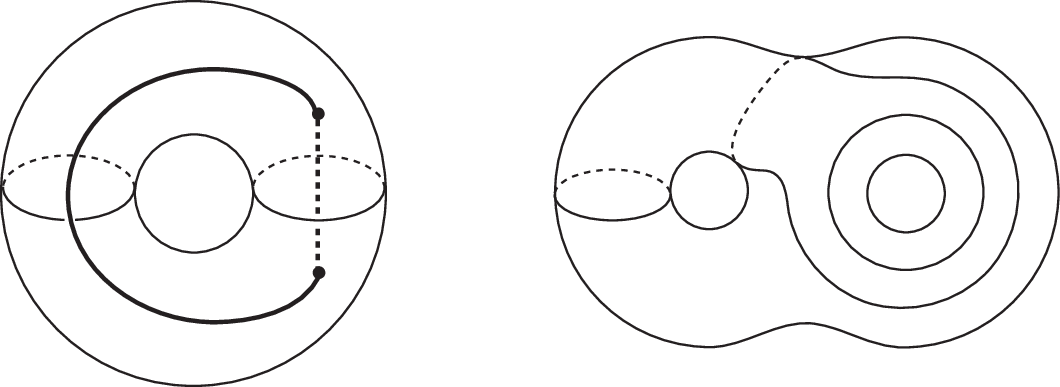}
\begin{picture}(400,0)(0,0)
\put(26,70){$m_1$}
\put(162,70){$m_2$}
\put(95,20){$K$}
\put(95,0){$V_1$}

\put(218,53){$m_1$}
\put(267,45){$m_2$}
\put(275,87){$C_{T_1}$}
\put(275,0){$\vtil$}

\end{picture}
\caption{The case where $m_1 \cap m_2 = \emptyset$ and $m_1 \neq m_2$.}
\label{figure:core_knot}
\end{figure}
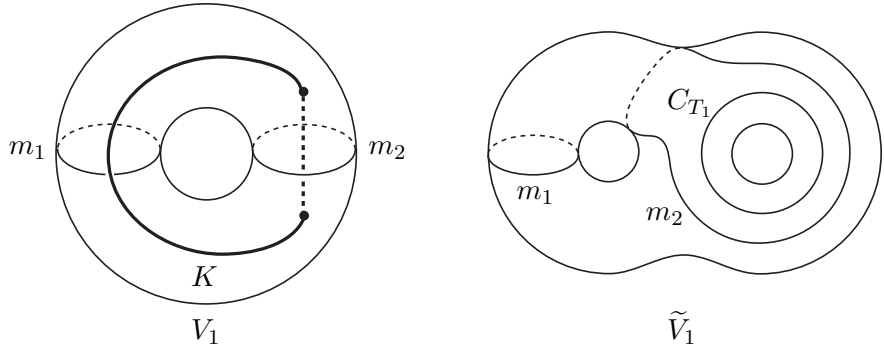

In the following, we compute $\MCG(\vtil, C_{T_1}, m_2)$ instead of $\MCG(V_1, T_1, m_2)$. 
By cutting $\partial \vtil$ along $m_1  \cup C_{T_1}$, we obtain a four-holed sphere 
$\Sigma'$, whose boundary components are $m_1^{\pm}$ and $C_{T_1}^{\pm}$. 
Then, $m_2$ remains to be a simple closed curve in $\Sigma'$, and, 
without loss of generality, we can assume that it separates 
$m_1^+ \cup C_{T_1}^+$ from ${m_1}^- \cup C_{T_1}^-$. 
Let $\Sigma'_1$ and $\Sigma'_2$ be the three-holed sphere obtained by cutting $\Sigma'$ along $m_2$, 
where $\partial \Sigma'_1 = m_1^+ \cup C_{T_1}^+ \cup m_2^+$. 
Since the group $\MCG(\Sigma'_1, m_1^+, C_{T_1}^+, m_2^+)$ is trivial 
(see e.g. Farb--Margalit~\cite[Proposition~2.3]{FM12}), 
the subgroup of $\MCG(\partial \vtil, C_{T_1}, m_2)$ consisting of elements that preserve $\Sigma'_1$ and $\Sigma'_2$ 
is generated by $\tau$, $\tau_{C_{T_1}}$, $\tau_{m_2}$. 
Noting that the map $\alpha$ interchanges $\Sigma'_1$ and $\Sigma'_2$, 
we see that $\MCG(\partial \vtil, C_{T_1}, m_2)$ is generated by $\alpha$, $\tau$, $\tau_{C_{T_1}}$ and $\tau_{m_2}$. 

The generators $\alpha$, $\tau$, $\tau_{C_{T_1}}$ and $\tau_{m_2}$ commute with each other. 
Furthermore, $\alpha$ and $\tau$ are elements of $\mathcal{G}(M, K; \Sigma)$ because they 
belong to $\MCG (\vtil, C_{T_1}, m_2)$. 
By McCullough~\cite[Theorem~$1$]{Mc}, it is known that 
a Dehn twist along a simple closed curve $c$ on the boundary of a compact orientable $3$-manifold $N$ 
extends to a homeomorphisms of  $N$ if and only if $c$ bounds a disk in $N$. 
Since neither $C_{T_1}$ nor $m_2$ bound a disk in $\vtil$, neither $\tau_{C_{T_1}}$ nor $\tau_{m_2}$ 
extends to a homeomorphism of $\vtil$. 
On the other hand, the map $\tau' := \tau \tau_{C_{T_1}} \tau_{m_2}^{-1}$ extends to a homeomorphism of $\vtil$. 
In fact, we can write $\tau' = \gamma_{\{D, E\}} \, \beta_D$, where $D$ and $E$ are canceling disks of $C_{T_1}$ shown in 
~\ref{figure:core_knot_DE}. 
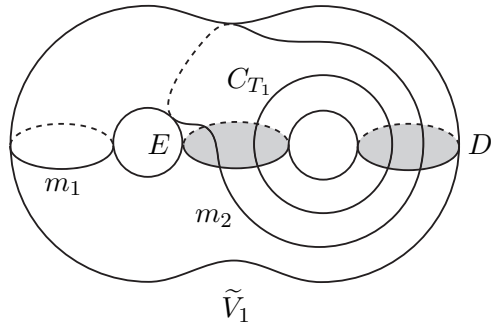
\begin{figure}[htbp]
\centering\includegraphics[width=6cm]{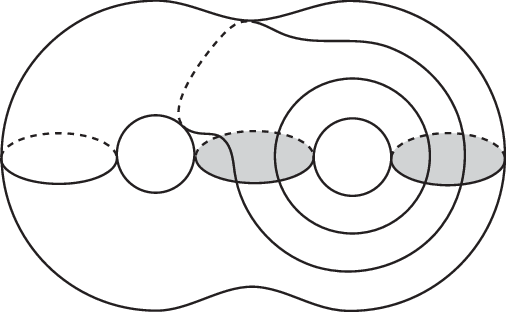}
\begin{picture}(400,0)(0,0)
\put(128,48){$m_1$}
\put(185,36){$m_2$}
\put(197,87){$C_{T_1}$}
\put(195,0){$\vtil$}
\put(167,62){$E$}
\put(288,62){$D$}
\end{picture}
\caption{The canceling disks $D$ and $E$ of $C_{T_1}$ in $\vtil$.}
\label{figure:core_knot_DE}
\end{figure}
Consequently, the Goeritz group $\mathcal{G}(M, K; \Sigma)$ has the following Smith normal form: 
$\mathcal{G}(M, K; \Sigma) = \MCG(V_1, T_1, m_2)  = \MCG(\vtil, C_{T_1}, m_2) 
= \ZZ / 2 \ZZ \langle \alpha \rangle \times \ZZ \langle \tau \rangle \times \ZZ \langle \tau' \rangle$.


\vspace{1em}

Now, we are ready to prove our main theorem. 
The proof will proceed case by case to establish the result.

\begin{proof}[Proof of Theorem $\ref{shukekka}$]

\noindent (1) This is Case~1-(a)-(i) (see Figure~\ref{figure:trivial_knot_lens}), thus we have $\mathcal{G}(M, K; \Sigma) = \ZZ / 2 \ZZ \langle \alpha \rangle \times \ZZ \langle \beta \rangle$.
\begin{figure}[htbp]
\centering\includegraphics[width=4.5cm]{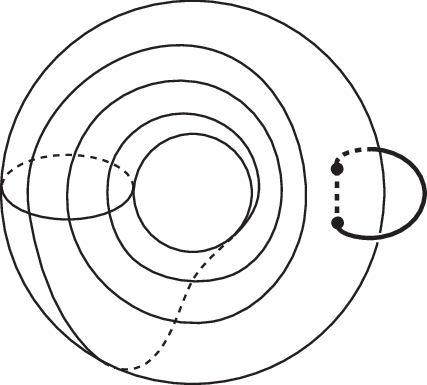}
\begin{picture}(400,0)(0,0)
\put(268,68){$K$}
\put(187,0){$V_1$}
\put(118,68){$m_1$}
\put(225,95){$m_2$}
\end{picture}
\caption{The trivial knot $K$ in $M = L(3, 1)$.}
\label{figure:trivial_knot_lens}
\end{figure}
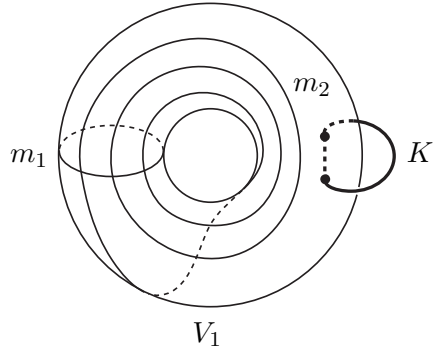

\noindent (2) 
In this case, $(S^2 \times S^1, K)$ is the double of $(V_1, T_1)$, and thus, $m_1$ coincides with $m_2$. 
This is Case~2-(a), thus we have 
$\mathcal{G}(M, K; \Sigma) = \ZZ / 2 \ZZ \langle \alpha \rangle \times \ZZ \langle \tau \rangle \times \langle \beta, \gamma \mid \gamma^2 = 1 \rangle$.

\noindent (3) 
By Lemma~\ref{coredouti}, we have $m_1 \cap m_2 = \emptyset$ and $m_1 \neq m_2$. 
This is Case~2-(b), thus we have 
$\mathcal{G}(M, K; \Sigma) = \ZZ / 2 \ZZ \langle \alpha \rangle \times \ZZ \langle \tau \rangle \times \ZZ \langle \tau' \rangle$.

\noindent (4) 
This is Case 1-(b)-(i), thus 
we have $\mathcal{G}(M, K; \Sigma) = \ZZ / 2 \ZZ \langle \alpha \rangle \times \ZZ / 2 \ZZ \langle \gamma \rangle$.

\noindent (5) 
Suppose first that $M \neq S^2 \times S^1$. 
Then, $m_1$ must intersect $m_2$. 
If there exist exactly two canceling disks of $C_{T_1}$ whose intersection with $m_2$ is minimal, 
it is Case 1-(b). 
Moreover, since it is not (4), it must be Case 1-(b)-(ii), hence, 
we have $\mathcal{G}(M, K; \Sigma) = \ZZ / 2 \ZZ \langle \alpha \rangle$. 
Suppose that there exists a unique canceling disk $D_1$ whose intersection with $m_2$ is minimal.  
Then we can show that $D_1 \cap m_2 \neq \emptyset$. 
In fact, suppose, for a contradiction, that $D_1 \cap m_2 = \emptyset$. 
Then we have $\beta_{D_1}^2 \in \mathcal{G}(M, K; \Sigma)$ because $\beta_{D_1}^2$ preserves $m_2$. 
In other words, the Dehn twist $\beta_{D_1}^2$ along $\ell_{D_1}$ extends to a homeomorphism of $V_2$. 
Thus, by McCullough~\cite[Theorem~$1$]{Mc}, the simple closed curve $\ell_{D_1}$ bounds a (separating) disk in $V_2$. 
In this case, we can easily show that $K$ is the trivial knot, which is (1). 
Thus, $D_1$ has non-empty intersection with  $m_2$, which is Case 1-(a)-(ii). 
Therefore, we have  $\mathcal{G}(M, K; \Sigma) = \ZZ / 2 \ZZ \langle \alpha \rangle$ as well.

Next, suppose that $M = S^2 \times S^1$. 
If $m_1 \cap m_2 \neq \emptyset$, then the same argument as above shows that 
$\mathcal{G}(M, K; \Sigma) = \ZZ / 2 \ZZ \langle \alpha \rangle$. 
Then, suppose that $m_1 \cap m_2 = \emptyset$. 
When $m_1 = m_2$, $K$ is the trivial knot, which is (2), and thus this case is excluded. 
When $m_1 \neq m_2$, $K$ is the core knot, which is (3), and thus this case is also excluded. 
This completes the proof of (5). 
\end{proof}

\section*{Acknowledgments} 
The authors would like to thank the anonymous referee
for their valuable comments and suggestions, 
which significantly improved the clarity and quality of the exposition.


\end{document}